\documentclass[a4paper,10pt,twoside]{article}

\usepackage[utf8]{inputenc}

\usepackage{amsmath}  
\usepackage{mathrsfs}  
\usepackage{amsthm}
\usepackage{amsfonts}
\usepackage{amssymb, latexsym}
\usepackage{mhequ}
\usepackage{verbatim}

\usepackage{kerkis}

\usepackage{graphicx}
\usepackage{color}
\usepackage{xcolor}
\usepackage{tikz}
\usetikzlibrary{calc,intersections,through,backgrounds}
\usetikzlibrary{decorations.pathreplacing}
\usetikzlibrary{shapes}

\usepackage{float} 
\usepackage{enumerate}
\usepackage[margin=1.5in, marginparwidth=1.2in]{geometry}

\usepackage[pdfauthor={}, pdftitle={}, pdftex]{hyperref}
\hypersetup{
  colorlinks=false,
  pdfborder={0 0 0}
  }

\usepackage{tocloft}
\usepackage{fancyhdr}

\usepackage[]{appendix}

\newcommand{\al}{\alpha}

\newcommand{\gm}{\gamma}
\newcommand{\eps}{\varepsilon}

\newcommand{\ee}{\mathrm{e}}
\newcommand{\dd}{\mathrm{d}}
\newcommand{\cc}{\mathrm{c}}

\newcommand{\RR}{\mathbb{R}}
\newcommand{\TT}{\mathbb{T}}

\newcommand{\C}{\mathcal{C}}

\newcommand{\Prob}{\mathbf{P}}
\newcommand{\Exp}{\mathbf{E}}

\newcommand{\lng}{\langle}
\newcommand{\rng}{\rangle}

\definecolor{darkred}{rgb}{0.5,0.0,0.0}

\theoremstyle{plain}
\newtheorem{theorem}{Theorem}[section]
\newtheorem*{theorem*}{Theorem}
\newtheorem{proposition}[theorem]{Proposition}
\newtheorem{corollary}[theorem]{Corollary}
\newtheorem{lemma}[theorem]{Lemma}

\theoremstyle{definition}

\newtheorem{assumption}[theorem]{Assumption}

\theoremstyle{remark}
\newtheorem{remark}[theorem]{Remark}

\numberwithin{equation}{section}

\title{Synchronisation by noise for the stochastic quantisation equation in dimensions $2$ and $3$}
\author{Benjamin Gess, Pavlos Tsatsoulis}
\date{}

\fancypagestyle{plain}{

  \fancyhf{}
  \fancyfoot[L]{\footnotesize \today}
  \fancyfoot[C]{\thepage}
}

\renewenvironment{abstract}
{
\begin{minipage}{.9\textwidth}\small\textbf{Abstract}.\noindent
}
{
\end{minipage}
}

\newenvironment{keywords}
{
\begin{minipage}{.9\textwidth}\small\textbf{Keywords}:\noindent
}
{
\end{minipage}
}

\newenvironment{msc}
{
\begin{minipage}{.9\textwidth}\small\textbf{MSC 2010}:\noindent
}
{
\end{minipage}
}

\begin{document}

\maketitle

\begin{abstract} 
We prove uniform synchronisation by noise with rates for the stochastic quantisation equation in dimensions two and three. The proof relies on a combination of
coming down from infinity estimates and the framework of order-preserving Markov semigroups derived in [Butkovsky, Scheutzow; 2019]. In particular, it is
shown that this framework can be applied to the case of state spaces given in terms of Hölder spaces of negative exponent.
\end{abstract}

\bigskip


\smallskip

\begin{msc} 60H15, 37A25, 35K57 \end{msc}

\tableofcontents

\section{Introduction}

In this work we prove uniform synchronisation by noise for the stochastic quantisation equation in space dimension
$d=2$ and $3$ given by
\begin{equs}
 \begin{cases}
  & (\partial_t - \Delta) u = - \left(u^3 - 3\infty u\right) + u + \xi \quad \text{on } (0,\infty)\times\TT^d\\
  &  u|_{t=0} = f,
 \end{cases}
  \label{eq:sAC}
\end{equs}
where $\xi$ is space-time white noise, the initial condition $f$ lies in a space of distributions and $\TT^d$ denotes the $d$-dimensional torus
of fixed, but arbitrarily large, size. More precisely, we prove the following theorem. 

\begin{theorem} \label{thm:main_synchr} Let $\al_0 = \mathbf{1}_{\{d=3\}} \frac{1}{2} + \theta$, for $\theta>0$ sufficiently small, 
and let $u(t;f)$ be the solution to \eqref{eq:sAC} at time $t\geq 0$ with initial condition $f$. Then, there exists $\lambda_*>0$ 
such that for every $\al\in(\al_0-\delta,\al_0]$, $0<\delta <\theta$, and $p> \frac{d}{\al-\al_0+\delta}$ integer, we have that
\begin{equs}
 \left(\Exp\left(\sup_{f_1,f_2\in \C^{-\al_0}} \|u(t;f_2) - u(t;f_1)\|_{-\al}\right)^p\right)^{\frac{1}{p}} 
 & \lesssim_p \ee^{-\frac{\lambda_*}{p} t}. \label{eq:intro-synchr}
\end{equs}
\end{theorem}

Here, for $\kappa>0$, $\C^{-\kappa}$ denotes a H\"older space of negative exponent with norm $\|\cdot\|_{-\kappa}$ (see Section \ref{s:notation} below for definitions).  

Theorem \ref{thm:main_synchr} will be obtained as a special case from a general framework implying quantified, \textit{uniform} synchronisation by noise, given in Theorem \ref{thm:synchr} below. 

We call \eqref{eq:intro-synchr} \textit{uniform} synchronisation by noise, since the speed at which two trajectories $u(t;f_2)$, $u(t;f_1)$ approach each other is uniform in the initial
conditions $f_1$, $f_2$. This is in contrast to the long-time behaviour of the deterministic analogon, given by
\begin{equs}
  \begin{cases}
   & (\partial_t - \Delta) u = - u^3 + u \quad \text{on } (0,\infty)\times\TT^d\\
   &  u|_{t=0} = f,
  \end{cases}
\end{equs}
which has finitely many unstable modes. In this sense, the inclusion of noise in \eqref{eq:sAC} produces global asymptotic stability.

As a corollary, we prove the existence of a weak attractor $\{\eta(0)\}$ consisting of a single random
point $\eta(0)$ with law given by the invariant measure of \eqref{eq:sAC}.

\begin{corollary}\label{cor:weak_attr} Under the assumptions of Theorem \ref{thm:main_synchr}, there exists a stationary family of random variables 
$\{\eta(t)\}_{t\geq 0}$ in $L^p(\Omega;\C^{-\al})$, with the law of $\eta(0)$ being the invariant measure for \eqref{eq:sAC}, 
such that 
\begin{equs}
 \left(\Exp \left(\sup_{f\in \C^{-\al_0}} \|u(t;f) - \eta(t)\|_{-\al}\right)^p\right)^{\frac{1}{p}} \lesssim_p \ee^{-\frac{\lambda_*}{p} t}.
\end{equs}
\end{corollary}

The proof of the main result heavily relies on the techniques derived in \cite{BS19}. In this context, the main insight of the present work is the realisation that the abstract condition 
posed in \cite[Theorem 2.3 (3)]{BS19} can be verified in Hölder spaces of negative exponent, a fact which may prove fruitful also in further context. Another important ingredient are the so-called 
``coming down from infinity'' estimates. In the present work these are used with the slight twist of constructing uniform upper and lower solutions. This allows to avoid the main assumption of
\cite[(1.3)]{FGS15} which could not be verified for singular SPDEs such as the stochastic quantisation equation in more than one dimension.


We now briefly comment on the existing literature. Synchronisation by noise for SPDEs has been investigated, for example, in \cite{ACW83,CCLR07,CCK07,G13}. Recently, some of these works have 
been generalised to weak synchronisation for SPDEs with \textit{multiplicative} noise in \cite{DGT19}. For the related effect of synchronisation in master-slave systems we refer to 
\cite{CS10} and the references therein. Approaches to synchronisation by noise based on local stability have been introduced in \cite{B91}, with extensions in \cite{FGS14,CGS16}, and large deviation techniques have been employed in \cite{MS88,MSS94,T08}. For
synchronisation for discrete time random dynamical systems (RDS) see \cite{H13,JK16,KJR13,N14} and the references therein. 


Synchronisation by noise for stochastic semi-flows, as in the present work, has been analysed, for example, in \cite{AC98,CCK07,C02,CS04,G13,G13-4}. A rather general framework of sufficient
conditions for synchronisation by noise for order-preserving systems has been given in \cite{FGS15}, which recently has been extended in \cite{BS19} to yield quantitative estimates. Synchronisation
for the KPZ equation has recently been shown in \cite{R19} also relying on order-preservation.


The ``coming down from infinity'' property for the stochastic quantisation equation first was obtained in the two dimensional case in \cite{TW18i}, where it was proven that after positive, but arbitrarily small,
time solutions are bounded in a suitable Besov space of negative exponent uniformly in the initial condition. The same result was obtained in three dimensions in \cite{MW17}.
Recently, in \cite{MW18} a stronger version of ``coming down from infinity'' was proven, which provides bounds on compact space-time sets. The implicit constant depends only on the realisation
of the noise on an enlargement of the set but it is uniform in the choice of space-time boundary conditions.    

Long time asymptotic behaviour for the stochastic quantisation equation has been studied in two dimensions in \cite{TW19} in the small noise and torus regime. In that work,
the authors prove that, with overwhelming probability, solutions started with initial conditions in a small neighbourhood of $1$ and $-1$ contract exponentially fast to each other,
with explicit estimates on the rate of contraction. This property can also be seen as a synchronisation-type result, but the method developed in \cite{TW19} heavily uses large deviation 
theory, hence it is restricted to small noise.

In the present work we are able to treat both two and three dimensions, without any restriction on the size of the noise and the size of the torus, and prove synchronisation uniformly over 
the initial conditions. In fact, in two dimensions, our proof also applies to higher polynomial non-linearities of odd degree with negative leading coefficient (see Remark \ref{rem:other_pol} below).

Let us now briefly comment on the solution theory for \eqref{eq:sAC}. The term $-3\infty u$ on the right hand side of \eqref{eq:sAC} is reminiscent of renormalisation, 
since \eqref{eq:sAC} is not well-posed in dimensions two and three otherwise. This can be made rigorous if we let $\xi_\eps = \xi*\rho_\eps$, where $\rho_\eps$ is a 
smooth mollifier in space. Then there exist (a family of) renormalisation constants $C_\eps\nearrow \infty$, as $\eps \nearrow 0$, such that the solution $u_\eps$ to
\begin{equs}
 \begin{cases}
  & (\partial_t - \Delta) u_\eps = - (u^3_\eps - 3C_\eps u_\eps) + u_\eps+ \xi_\eps \\
  &  u_\eps|_{t=0} = f*\rho_\eps,
 \end{cases}
 \label{eq:rsAC}
\end{equs}
converges to a distribution $u$. Moreover, under the right choice of renormalisation constants $C_\eps$ the limit 
is independent of the choice of $\rho_\eps$. We refer the reader to \cite{DPD03} and \cite[Section 9.4]{Ha14} for 
more details on how one can choose $C_\eps$ to obtain a (non-trivial) limit of the solution $u_\eps$ to \eqref{eq:rsAC}.
Here we study the limit $u:=\lim_{\eps\searrow0}u_\eps$ as a stochastic semi-flow taking values in a H\"older space
of negative exponent.  

\subsection{Notation and basic definitions} \label{s:notation}

For $p\in[1,\infty]$, we denote by $\|\cdot\|_p$ the usual $L^p$-norm on the space of periodic functions $f:\TT^d\to \RR$. 
For $\al>0$ and $p\in[1,\infty)$, we define the Besov norm $\|\cdot\|_{-\al;p}$ on the space of periodic distributions $f$ by
\begin{equs}
 \|f\|_{-\al;p} & := \left(\int_0^1 s^{\frac{\al p}{2}} \|f_s\|_p^p \frac{\dd s}{s}\right)^{\frac{1}{p}}, \label{eq:besov_neg_p}
\end{equs}
where for $s\in(0,1]$ we denote by $f_s$ the convolution with the heat kernel, that is, $f_s := \ee^{s\Delta}f$. We also define
the Besov norm $\|\cdot\|_{-\al}$ by
\begin{equs}
 \|f\|_{-\al} & := \sup_{s\in(0,1]} s^{\frac{\al}{2}} \|f_s\|_\infty \label{eq:besov_neg_infty}
\end{equs}
and we denote by $\C^{-\al}$ the closure of $\C^\infty$ with respect to $\|\cdot\|_{-\al}$. Note that $\C^{-\al} \subset \C^{-\beta}$
if $\al\leq \beta$ and, in particular, $\|\cdot\|_{-\beta} \lesssim _{\al,\beta} \|\cdot\|_{-\al}$.

\begin{remark} Let us mention that definitions \eqref{eq:besov_neg_p} and \eqref{eq:besov_neg_infty} are equivalent to the more standard definitions of Besov norms through Littlewood--Paley blocks 
(see for example \cite[Theorem 2.34]{BCD11}). Using the same idea as in \cite[proof of Lemma A.3-i]{OW19}, one can 
replace the semigroup $\ee^{s\Delta}$ in \eqref{eq:besov_neg_p} and \eqref{eq:besov_neg_infty} by any rescaled smooth function 
$\psi_s(x) = s^{-\frac{d}{2}} \psi(s^{-\frac{1}{2}}x)$ such that $\int_{\TT^d} \psi \, \dd x =1$. 
\end{remark}

It is immediate from \eqref{eq:besov_neg_p} and \eqref{eq:besov_neg_infty} that for every $p\geq 1$ the following holds,
\begin{equs}
 \|f\|_{-\al;p} & \lesssim_{\al,p} \|f\|_{-\al}. \label{eq:besov_emb_infty_p} 
\end{equs}
For every $\al\in \RR$ and $p\geq1$, we also have that
\begin{equs}
 \|f\|_{-\al} & \lesssim_{\al,p} \|f\|_{-\al+\frac{d}{p};p}. \label{eq:besov_emb_p_infty}
\end{equs}
This inequality is essentially \cite[Proposition 2.20]{BCD11}. It will be useful in the sequel since it allows us to pass from \eqref{eq:besov_neg_infty}
to \eqref{eq:besov_neg_p}. 

For functions $f,g$ we write $f\preceq g$ whenever $f(x)\leq g(x)$ for almost every $x$. When $f,g$ are periodic distributions we
also write $f\preceq g$ whenever $\lng f,\varphi\rng \leq \lng g, \varphi\rng$ for every non-negative $\varphi\in \C^\infty$. 

Let $(\Omega, \mathcal{F}, \Prob)$ be a probability space, $(\mathcal{F}_{s,t})_{t\geq s\geq 0}$ be a family of sub-$\sigma$-algebras 
of $\mathcal{F}$ and $(\mathcal{M},\mathfrak{m})$ be a metric space endowed with the Borel $\sigma$-algebra. We call a collection 
$\{u(t;s;\cdot); t\geq s\geq 0\}$ of $\mathcal{F}_{s,t}$-measurable random maps
\begin{equs}
 u(t;s;\cdot)(\omega):  \C^{-\al_0} \ni f \mapsto u(t;s;f)(\omega) \in \C^{-\al_0}, \quad \omega\in \Omega,
\end{equs}
a white noise stochastic semi-flow on $\mathcal{M}$ if the following holds, 
\begin{enumerate}
 \item The $\sigma$-algebras $\mathcal{F}_{s,t}$ and $\mathcal{F}_{s',t'}$ are independent for $s\leq t\leq s'\leq t'$.
 \item $\{u(t;s;\cdot); t\geq s\geq 0\}$ satisfies the flow property, that is, 
 $u(t+s;0;f)(\omega) = u(t;s;u(s;0;f)(\omega))(\omega)$, for every $t\geq s\geq 0$, $f\in \mathcal{M}$ and $\omega\in \Omega$.
 \item For every $f\in \C^{-\al_0}$ and $t\geq s\geq 0$, $u(t;s;f)$ and $u(t-s;0;f)$ have the same law. 
\end{enumerate}

For probability measures $\mu,\nu$ on a metric space $(\mathcal{M},\mathfrak{m})$ and $p\geq 1$, we define the Wasserstein distance 
$\mathcal{W}_{\mathfrak{m};p}(\mu,\nu)$ by
\begin{equs}
 \mathcal{W}_{\mathfrak{m};p}(\mu,\nu) := \left(\inf_{\pi \in \mathscr{C}(\mu,\nu)} \int_{\mathcal{M}\times\mathcal{M}} \mathfrak{m}(k,l)^p \, \pi(\dd k, \dd l)\right)^{\frac{1}{p}},
\end{equs}
where $\mathscr{C}(\mu,\nu)$ is the space of all probability measures $\pi$ which form a coupling of $\mu$ and $\nu$.

\subsection*{Acknowledgements}

Financial support by the DFG through the CRC 1283 ``Taming uncertainty and profiting from randomness and low regularity in analysis, stochastics and their applications''
is acknowledged.

\section{General framework}

In this section we work on a probability space $(\Omega, \mathcal{F}, \Prob)$ with a family $(\mathcal{F}_{s,t})_{t\geq s\geq 0}$ of sub-$\sigma$-algebras 
of $\mathcal{F}$. For $\al_0>0$, we fix an $\mathcal{F}_{s,t}$-measurable white noise stochastic semi-flow $\{u(t;s;\cdot): t\geq s\geq0\}$ on $\C^{-\al_0}$. Furthermore, we
assume that there exists $\delta>0$ such that the following holds. 

\begin{assumption} \label{ass:cont} For every $t> s\geq 0$ and every $\al\in(\al_0-\delta, \al_0]$ we have that
\begin{equs}
 \C^{-\al_0} \ni f \mapsto u(t;s;f)(\omega) \in \C^{-\al}, \quad \omega\in \Omega,
\end{equs}
and this map is continuous in $f$.
\end{assumption}

\begin{assumption} \label{ass:order_pres} The stochastic semi-flow is order preserving with respect to $\preceq$, that is,
whenever $f\preceq g$ we have that $u(t;s;f)(\omega) \preceq u(t;s;g)(\omega)$ for every $t\geq s \geq 0$ and $\omega \in \Omega$. 
\end{assumption}

\begin{assumption} \label{ass:coming_down_infty} For every $s\geq 0$ there exists a non-negative random variable $K_s$ such that $\sup_{s\geq0}\Exp K_s^p<\infty$, for every $p\geq 1$,
and the following bound holds for some $\gm>0$,
\begin{equs}
 \sup_{f\in\C^{-\al_0}} \sup_{t\in[s,s+1]} (t-s)^\gm \|u(t;s;f)(\omega)\|_{-\al} \leq K_s(\omega), \quad \omega\in \Omega,
\end{equs}
for every $\al\in(\al_0-\delta, \al_0]$. 
\end{assumption}

\begin{assumption} \label{ass:ergod_wass} There exists $\lambda_*>0$ such that for every $\al\in(\al_0-\delta,\al_0]$ and every $p\geq 1$,
\begin{equs}
 \sup_{f_1,f_2\in \C^{-\al_0}} \mathcal{W}_{\|\cdot-\cdot\|_{-\al;p}\wedge 1;p}\big(u(t;0;f_1), u(t;0;f_2)\big) \lesssim_{\al,p} \ee^{-\lambda_* t}.
\end{equs}
\end{assumption}

With these assumptions at hand we have the following theorem.

\begin{theorem} \label{thm:synchr} Let $\{u(t;s;\cdot):t\geq s\geq0\}$ be a white noise stochastic semi-flow which satisfies Assumptions \ref{ass:cont}-\ref{ass:ergod_wass}. Then, 
for every $\al\in (\al_0-\delta,\al_0]$ and $p> \frac{d}{\al-\al_0+\delta}$ integer, we have that 
\begin{equs}
 \left(\Exp\left(\sup_{f_1,f_2\in \C^{-\al_0}} \|u(t;0;f_2) - u(t;0;f_1)\|_{-\al}\right)^p\right)^{\frac{1}{p}} 
 & \lesssim_p \ee^{-\frac{\lambda_*}{p} t}. \label{eq:synchr}
\end{equs}
\end{theorem}

\begin{proof} To ease the notation we will drop the dependence on $\omega$. We will also simply write $u(t;f)$ instead of $u(t;0;f)$. 
The proof of the theorem consists of two steps. 

\textit{Step 1}: We fix $\al\in(\al_0-\delta,\al_0]$. We first construct profiles $u_+, u_-\in \C([1,\infty];\C^{-\al})$
with the followings properties, 
\begin{enumerate}[i.]
 \item \label{it:extr_order} For every $u_0\in \C^{-\al_0}$ and every $t\geq 1$, $u_-(t) \preceq u(t;u_0) \preceq u_+(t)$. 
 \item \label{it:extr_bnd} For every $p> \frac{d}{\al-\al_0+\delta}$, $\sup_{t\geq 2}\left(\Exp\|u_{\pm}(t)\|_{-\al+\frac{d}{p};p}^{\frac{p^2}{p-1}}\right)^{\frac{p-1}{p}} <\infty$. 
 \item \label{it:extr_ergod} For every $p> \frac{d}{\al-\al_0+\delta}$, 
  \begin{equs}
    \mathcal{W}_{\|\cdot-\cdot\|_{-\al+\frac{d}{p};p}\wedge 1;p}\big(u_+(t), u_-(t)\big) \lesssim_{\al,p} \ee^{-\lambda_* t},
    \label{eq:ergod_extr}
  \end{equs}
 for every $t\geq 1$. 
\end{enumerate}

For $R>0$ let $u_{\pm R}(1)$ be given by $u(1;\pm R)$. By Assumption \ref{ass:coming_down_infty} we know that for every $\eps>0$ sufficiently small, 
\begin{equs}
 \sup_{R>0} \|u_{\pm R}(1)\|_{-\al+\eps} \leq K.
\end{equs}
Using the fact that $\C^{-\al+\eps}$ embeds compactly into $\C^{-\al}$, there exists a sequence $R_n\nearrow \infty$ and $u_{\pm}(1)\in\C^{-\al}$
such that $u_{R_n}(1) \to u_{\pm}(1)$ in the sense of distributions. 

We now prove that $u_{\pm R} (1)\to u_{\pm}(1)$ in the sense of distributions. Let $\tilde R_n\nearrow \infty$ be another sequence, fix $\varphi\in \C^\infty$ non-negative and set $a_n = \lng u_{R_n}, \varphi \rng$, $b_n = \lng u_{\tilde R_n}, \varphi \rng$.
By Assumptions \ref{ass:order_pres} and \ref{ass:coming_down_infty} we know that $a_n$ and $b_n$ are non-decreasing and bounded, hence there exist $a$ and $b$ such that $a_n \nearrow a$ and $b_n \nearrow b$. 
Using again Assumption \ref{ass:order_pres} we see that for every $n\geq 1$ there exists $n_0,n_0'\geq n$ such that $a_{n_0}\geq b_n$ and $b_{n_0'} \geq a_n$ which implies that $a=b$. 
The same holds if $\varphi\in\C^\infty$ is arbitrary since we can find $K_\varphi>0$ such that $\varphi \geq -K_\varphi$ and use that
\begin{equs}
 \lim_{n\to\infty} \lng u_{R_n}, \varphi \rng & = \lim_{n\to\infty} \lng u_{R_n}, \varphi + K_\varphi \rng - \lim_{n\to\infty} \lng u_{R_n}, K_\varphi \rng 
 \\
 & = \lim_{n\to\infty} \lng u_{\tilde R_n}, \varphi + K_\varphi \rng - \lim_{n\to\infty} \lng u_{\tilde R_n}, K_\varphi \rng 
 \\
 & = \lim_{n\to\infty} \lng u_{\tilde R_n}, \varphi \rng.
\end{equs}
Therefore, the limit $u_+(1)$ is independent of the choice of the sequence $R_n$ and we have that $u_R(1) \to u_+(1)$ in the sense of distributions. The same argument implies
that $u_{-R}(1) \to u_-(1)$ in the sense of distributions. 

To construct $u_{\pm}(t)$ we first notice that for $f\in \C^{-\al_0}$, if we let $f^{(\eps)} := f*\rho_\eps$, where $\rho_\eps$, $\eps\in(0,1]$, is a
smooth mollifier, then $f^{(\eps)}\to f$ in $\C^{-\al_0}$, as $\eps \searrow 0$, and $\|f^{(\eps)}\|_\infty \leq R_\eps$ for some 
$R_\eps\nearrow \infty$. By Assumption \ref{ass:order_pres} we know that
\begin{equs}
 u_{-R_\eps}(1) \preceq u(1;f^{(\eps)}) \preceq u_{R_\eps}(1),
\end{equs}
and letting $\eps\searrow 0$, by Assumption \ref{ass:cont} we get that 
\begin{equs}
 u_-(1) \preceq u(1;f) \preceq u_+(1).
\end{equs}
Using once more Assumption \ref{ass:order_pres} and the flow property we obtain that
\begin{equs}
 u(t;1;u_-(1)) \preceq u(t;f) \preceq u(t;1;u_+(1)) \label{eq:compar}
\end{equs}
for every $t\geq 1$. We now set $u_\pm(t) := u(t;1;u_\pm(1))$. Hence property \ref{it:extr_order} is satisfied due to \eqref{eq:compar}, 
while \ref{it:extr_bnd} and \ref{it:extr_ergod} are essentially Assumptions \ref{ass:coming_down_infty} and \ref{ass:ergod_wass}
combined with \eqref{eq:besov_emb_infty_p} and the fact that $\mathcal{F}_{0,1}$ and $\mathcal{F}_{1,t}$ are independent for $t\geq 1$.

Before we proceed to the second step let us introduce the following notation. For $\al>0$, $p>\frac{d}{\al}$ and periodic distribution $f$, 
we let  
\begin{equs}
 \varphi_{-\al+\frac{d}{p};p}(f) = \int_0^1 s^{\frac{(\al-\frac{d}{p})p}{2}} \varphi_p(f_s) \frac{\dd s}{s},
\end{equs}
where 
\begin{equs}
 \varphi_p(f) = 2^{p-1} \int \left(f^p \mathbf{1}_{\{f\geq 0\}} - |f|^p \mathbf{1}_{\{f<0\}}\right) \, \dd x 
 = 2^{p-1} \int \mathrm{sgn}(f) |f|^p \, \dd x. 
\end{equs}
It is easy to see that $\varphi_{-\al+\frac{d}{p};p}(f)$ is finite whenever $\|f\|_{-\al+\frac{d}{p};p}$ is finite and, in particular, we have that
\begin{equs}
 \varphi_{-\al;p}(f) & \leq 2^{p-1} \|f\|_{-\al;p}^p. 
\end{equs}

\textit{Step 2}: Let $f_1,f_2\in \C^{-\al_0}$. By Step 1 we know that
\begin{equs}
 & u_+(t) \preceq u(t;f_i) \preceq u_-(t)
\end{equs}
for every $t\geq 1$ and $i=1,2$. We then notice that 
\begin{equs}
 \|u(t;f_2) - u(t;f_1)\|_{-\al}^p & \stackrel{\eqref{eq:besov_emb_p_infty}}{\lesssim} \|u(t;f_2) - u(t;f_1)\|_{-\al+\frac{d}{p};p}^p \\
 & \lesssim \|u(t;f_2) - u_-(t)\|_{-\al+\frac{d}{p};p}^p + \|u(t;f_1) - u_-(t)\|_{-\al+\frac{d}{p};p}^p,
\end{equs}
while, by Lemma \ref{lem:norm_decoupl}, we know that
\begin{equs}
 \|u(t;f_i) - u_-(t)\|_{-\al+\frac{d}{p};p}^p & \lesssim \varphi_{-\al+\frac{d}{p};p}(u(t;f_i)) - \varphi_{-\al+\frac{d}{p};p}(u_-(t)) \\
 & \lesssim \varphi_{-\al+\frac{d}{p};p}(u_+(t)) - \varphi_{-\al+\frac{d}{p};p}(u_-(t)),
\end{equs}
where we also use that $\varphi_{-\al+\frac{d}{p};p}$ is non-decreasing with respect to $\preceq$. Hence, we have proved that
\begin{equs}
 \left(\sup_{f_1,f_2\in \C^{-\al_0}} \|u(t;f_2) - u(t;f_1)\|_{-\al}\right)^p 
 \lesssim \varphi_{-\al+\frac{d}{p};p}(u_+(t)) - \varphi_{-\al+\frac{d}{p};p}(u_-(t)) 
 \\
 \label{eq:varphi_contr}
\end{equs}
for every $t\geq 1$. We now choose a coupling $(\tilde u_+(t),\tilde u_-(t))$ of $u_+(t)$ and $u_-(t)$ such that
\begin{equs}
 \left(\Exp_{\mathrm{c}} \left(\|\tilde u_+(t)-\tilde u_-(t)\|_{-\al+\frac{d}{p};p}\wedge1\right)^p\right)^{\frac{1}{p}} 
 & \lesssim \ee^{-\lambda_* t}.
\end{equs}
This is possible due to \eqref{eq:ergod_extr}. By Lemma \ref{lem:phi_bnd}, we know that
\begin{equs}
 & \Exp_{\mathrm{c}}\left(\varphi_{-\al+\frac{d}{p};p}(\tilde u_+(t)) - \varphi_{-\al+\frac{d}{p};p}(\tilde u_-(t))\right) 
 \\
 & \quad \lesssim \left(\Exp_{\mathrm{c}}(\|\tilde u_+(t) - \tilde u_-(t)\|_{-\al+\frac{d}{p};p}\wedge 1)^p\right)^{\frac{1}{p}} 
 \max\left\{1,\left(\Exp_{\mathrm{c}} \|\tilde u_\pm(t)\|_{-\al+\frac{d}{p};p}^{\frac{p^2}{p-1}}\right)^{\frac{p-1}{p}}\right\}. 
\end{equs}
Combining this estimate with property \ref{it:extr_bnd} we get that 
\begin{equs}
 \Exp_{\mathrm{c}}\left(\varphi_{-\al+\frac{d}{p};p}(\tilde u_+(t)) - \varphi_{-\al+\frac{d}{p};p}(\tilde u_-(t))\right)
 & \lesssim \ee^{-\lambda_* t},
\end{equs}
for every $t\geq 2$. To conclude, we notice that
\begin{equs}
 \Exp \left(\sup_{f_1,f_2\in \C^{-\al_0}} \|u(t;f_2) - u(t;f_1)\|_{-\al}\right)^p & \lesssim 
 \Exp \varphi_{-\al+\frac{d}{p};p}(u_+(t)) - \Exp \varphi_{-\al+\frac{d}{p};p}(u_-(t)) 
 \\
 & = \Exp_\cc \varphi_{-\al+\frac{d}{p};p}(\tilde u_+(t)) - \Exp_\cc \varphi_{-\al+\frac{d}{p};p}(\tilde u_-(t))
 \\ 
 & \lesssim \ee^{-\lambda_* t}.
\end{equs}
\end{proof}

\section{Applications to the stochastic quantisation equation}

As we already discussed in the introduction, the solutions to the stochastic quantisation equation \eqref{eq:sAC} can be interpreted in a renormalised
sense as limits of \eqref{eq:rsAC} where, in particular, one has to choose $C_\eps = \mathbf{1}_{\{d=3\}} C_3 \eps^{-1} + C_2 \log\eps^{-1} + C_1$ for 
suitable $C_i>0$, $i=1,2,3$. Then, for $\theta > 0$ sufficiently small and $\al_0 = \mathbf{1}_{\{d=3\}} \frac{1}{2} + \theta$, for every $f\in \C^{-\al_0}$
the solution $u_\eps(\cdot;f*\rho_\eps)$ to \eqref{eq:rsAC} converges to a limit $u(\cdot;f)$ in $\C^{-\al}$, for every $\al\in(\al_0-\delta,\al_0]$ and 
$0<\delta<\theta$, uniformly on compact subsets of $(0,\infty)$, $\Prob$-almost surely (see \cite[Proposition 2.3 and Proposition 3.4]{TW18i} 
for $d=2$ and \cite[Theorem 1.15]{Ha14} for $d=3$). The same statement holds if we consider $u(\cdot;s;f)$ for $s>0$,
that is, the limit, as $\eps \searrow 0$, of the solution $u_\eps(\cdot;s;f*\rho_\eps)$ to 
\begin{equs}
 \begin{cases}
  & (\partial_t - \Delta) u_\eps(\cdot;s) = - \left(u_\eps(\cdot;s)^3 - 3C_\eps  u_\eps(\cdot;s)\right) +  u_\eps(\cdot;s)+ \xi_\eps \\
  & u_\eps(\cdot;s)|_{t=s} = f*\rho_\eps.
 \end{cases}
 \\
 \label{eq:rsAC_restart}
\end{equs}
Furthermore, one has that $\Prob$-almost surely, for every $f\in \C^{-\al_0}$,  
\begin{equs}
 u_\eps(t;0;f*\rho_\eps) = u_\eps(t;s;u_\eps(s;0;f*\rho_\eps))
\end{equs}
for every $t\geq s\geq0$ and $\eps\in(0,1]$. 

\begin{proof}[Proof of Theorem \ref{thm:main_synchr}] The fact that the family $\{u(t;s;\cdot);t\geq s\geq 0\}$ gives rise to a white noise stochastic semi-flow follows 
from the fact that it arises as the $\Prob$-almost sure limit of the solution to \eqref{eq:rsAC_restart} and the corresponding family of 
sub-$\sigma$-algebras is given by the usual augmentation $(\mathcal{F}_{s,t})_{t\geq s\geq 0}$ of
\begin{equs}
 \tilde{\mathcal{F}}_{s,t} = \sigma\left(\{\xi(\phi):\phi\in L^2(\RR\times\TT^2) \text{ such that } \phi|_{[0,s)\cup(t,\infty)\times \TT^d} \equiv 0\}\right), \quad t\geq s\geq 0.
\end{equs}
To prove Theorem \ref{thm:main_synchr}, we need to verify Assumptions \ref{ass:cont}-\ref{ass:ergod_wass}.

\smallskip

Assumption \ref{ass:cont}: This is immediate from \cite[Proposition 4.3]{TW18i} for $d=2$ and \cite[Proposition 6.9 and Theorem 7.8]{Ha14} for $d=3$.

\smallskip

Assumption \ref{ass:order_pres}: We notice that by Proposition \ref{prop:order_pres} $\{u_\eps(t;s;\cdot):t\geq s\geq 0\}$ is order preserving for every $\eps>0$. Without loss of generality, 
we can assume that $\rho_\eps$ is a non-negative smooth mollifier, so that if $f_1\preceq f_2$ then $f_1*\rho_\eps \preceq f_2*\rho_\eps$. Since $u_\eps(t;s;f_i*\rho_\eps) \to u(t;s;f_i)$ in $\C^{-\al}$
for $i=1,2$, as $\eps \searrow 0$, we also have that $\lng u_\eps(t;s;f_i*\rho_\eps), \varphi \rng \to \lng u(t;s;f_i), \varphi \rng$, for every non-negative $\varphi\in\C^\infty$. This implies order
preservation for $\{u(t;s;\cdot):t\geq s\geq 0\}$. 

\smallskip

Assumption \ref{ass:coming_down_infty}: This was verified in \cite[Corollary 3.10]{TW18i} for $d=2$ and \cite[Eq. 1.27]{MW17} for $d=3$. 

\smallskip

Assumption \ref{ass:ergod_wass}: This was verified in \cite[Theorem 6.5]{TW18i} for $d=2$. Actually, in that work exponential mixing was obtained in the total variation distance which implies exponential mixing
in the Waserstein distance. The same result also holds for $d=3$ and it was obtained recently in \cite[Corollary 1.9]{HS19}.  
\end{proof}

\begin{proof}[Proof of Corollary \ref{cor:weak_attr}] We can extend our time interval to $-\infty$ (by adapting the definition of $\mathcal{F}_{s,t}$ for $t\geq s >-\infty$) 
and let $\eta(t;s):= u(t;s;0)$, for $t\geq s\geq -\infty$. Then, for $0\geq s_1\geq s_2$, by the flow property we have that 
\begin{equs}
 \eta(t;s_2) - \eta(t;s_1) = u(t;s_1;u(s_1;s_2;0)) - u(t;s_1;0).
\end{equs}
By Theorem \ref{thm:main_synchr} (extended for negative initial times) and the fact that $\mathcal{F}_{s_2,s_1}$ and 
$\mathcal{F}_{s_1,t}$ are independent we know that
\begin{equs}
 & \Exp\|u(t;s_1;u(s_1;s_2;0)) - u(t;s_1;0)\|_{-\al}^p 
 \\
 & \quad = \Exp \left(\Exp\|u(t;s_1;u(s_1;s_2;0)) - u(t;s_1;0)\|_{-\al}^p\big|\mathcal{F}_{s_1,s_2}\right) 
 \\
 & \quad = \Exp \left(\Exp\|u(t;s_1;f) - u(t;s_1;0)\|_{-\al}^p\right)\big|_{f=u(s_1;s_2;0)} 
 \\
 & \quad \leq \Exp\left(\sup_{f_1,f_2\in \C^{-\al_0}} \|u(t;s_1;f_2) - u(t;s_1;f_1)\|_{-\al}\right)^p 
 \\
 & \quad \lesssim \ee^{-\lambda_* (t-s_1)}.
\end{equs}
This implies that the sequence $(\eta(t;s))_{s\leq0}$ is Cauchy in $L^p(\Omega;\C^{-\al})$. Hence there exists
$\eta(t)\in L^p(\Omega;\C^{-\al})$ such that $\eta(t;s) \to \eta(t)$ in $L^p(\Omega;\C^{-\al})$, as $s\searrow -\infty$, 
and furthermore we have the estimate
\begin{equs}
 \left(\Exp\|\eta(t;s) - \eta(t)\|_{-\al}^p\right)^{\frac{1}{p}} & \lesssim \ee^{-\frac{\lambda}{p} (t-s)}. 
\end{equs}
Using again Theorem \ref{thm:main_synchr} it is easy to see that
\begin{equs}
  \left(\Exp\sup_{f\in \C^{-\al_0}}\|u(t;s;f) - \eta(t)\|_{-\al}^p\right)^{\frac{1}{p}} 
  & \lesssim \ee^{-\frac{\lambda}{p} (t-s)}, 
\end{equs}
which in particular holds for $s=0$. To prove that $\eta(t)$ is stationary in $t$ it suffices to notice that
since $\{u(t;s;\cdot);t\geq s >-\infty\}$ is a white noise stochastic semi-flow, we have that
\begin{equs}
 u(t;s;0) = u(t-s;0;0) = u(0;-(t-s);0),
\end{equs}
where the above equalities hold in law. Since $u(t;s;0)\to \eta(t)$ and $u(0;-(t-s);0)\to \eta(0)$ in $L^p(\Omega;\C^{-\al})$, as $s\searrow -\infty$,
we obtain that $\eta(t)$ is stationary. It is also easy to see that the law of $\eta(0)$ is invariant for \eqref{eq:sAC}. Indeed, for every
continuous bounded function $F$ on $\C^{-\al}$ we have that
\begin{equs}
 \Exp\big( \Exp F(u(t;f))\big)\big|_{f=\eta(0)} & =
 \lim_{s\to-\infty}\Exp\big( \Exp F(u(t;f))\big)\big|_{f=u(0;s;0)}
 = \lim_{s\to-\infty}\Exp F(u(t;s;0)) 
 \\
 & = \lim_{s\to-\infty}\Exp F(u(0;-(t-s);0))  
 = \Exp F(\eta(0))
\end{equs}
which proves the claim. 
\end{proof}

\begin{remark} \label{rem:other_pol} When $d=2$ we can actually consider any odd polynomial non-linearity with negative leading coefficient. In particular, as 
in \cite{DPD03}, we can study the limit $u$ of
\begin{equs}
 \begin{cases}
  & (\partial_t - \Delta) u_\eps = - \sum_{k=1}^n a_k \mathcal{H}_k(u_\eps, C_\eps) + \xi_\eps \\
  & u_\eps|_{t=0} = f*\rho_\eps,
 \end{cases}
\end{equs}
where $n$ is odd and $a_n>0$. Here $\mathcal{H}_k(u,C)$ stands for the $k$-th Hermite polynomial. Combining \cite[Corollary 3.10]{TW18i} 
and \cite[Corollary 1.9]{HS19} the proof of Theorem \ref{thm:main_synchr} immediately applies to this case. 
\end{remark}

\begin{appendix}

\section{Control of negative Besov norms}

Below, for $p\geq 1$ and function $f\in L^p$ we let 
\begin{equs}
 \varphi_p(f) = 2^{p-1} \int \left(f^p \mathbf{1}_{\{f\geq 0\}} - |f|^p \mathbf{1}_{\{f<0\}}\right) \, \dd x = 
 2^{p-1} \int \mathrm{sgn}(f) |f|^p \, \dd x.
\end{equs}

In \cite[Example 2.7]{BS19} the functional $\varphi_p$ was introduced to provide a bound on the difference of two ordered functions $g\preceq f$ in the 
$L^p$-norm, that is, 
\begin{equs}
 \|f-g\|_{L^p}^p & \leq \varphi_p(f) - \varphi_p(g).
\end{equs}
In the same spirit one can built a functional $\varphi_{-\al;p}$ which allows to bound the difference of two ordered distributions $g\preceq f$ in the Besov
norm $\|\cdot\|_{-\al;p}$, for $p\geq 1$ and $\al>0$, as we prove in the following lemma. 


\begin{lemma} \label{lem:norm_decoupl} Let $\al>0$ and $p\geq 1$. For periodic distributions $g\preceq f$ such that $\|f\|_{-\al;p}, \|g\|_{-\al;p}<\infty$ the 
following estimate holds,
\begin{equs}
 \|f-g\|_{-\al;p}^p & \leq \int_0^1 s^{\frac{\al p}{2}} \varphi_p(f_s) \frac{\dd s}{s} - \int_0^1 s^{\frac{\al p}{2}} \varphi_p(g_s) \frac{\dd s}{s}.
\end{equs}
\end{lemma}

\begin{proof} Since $g\preceq f$, for every $s\in (0,1]$ we also have that $g_s\leq f_s$. By \cite[Proof of Example 2.7]{BS19} we know that  
\begin{equs}
 \|f_s-g_s\|_p^p & \leq \varphi_p(f_s) - \varphi_p(g_s)
\end{equs}
which in turn implies that
\begin{equs}
 \|f-g\|_{-\al;p}^p = \int_0^1 s^{\frac{\al p}{2}} \|f_s-g_s\|_{L^p}^p \frac{\dd s}{s} & \leq \int_0^1 s^{\frac{\al p}{2}} (\varphi_p(f_s) - \varphi_p(g_s)) \frac{\dd s}{s}.
\end{equs}
\end{proof}

\begin{lemma} \label{lem:phi_bnd} Let $\al>0$ and $p\geq 1$ integer. For periodic distributions $f,g$ such that $\|f\|_{-\al;p}, \|g\|_{-\al;p}<\infty$ the following
estimate holds,
\begin{equs}
 & \left|\int_0^1 s^{\frac{\al p}{2}} \varphi_p(f_s) \frac{\dd s}{s} - \int_0^1 s^{\frac{\al p}{2}} \varphi_p(g_s) \frac{\dd s}{s}\right| \\
 & \quad \lesssim_{\al,p} \left(\|f-g\|_{-\al;p}\wedge 1\right) \, 
 \max\left\{1, \|f\|_{-\al;p}^p + \|g\|_{-\al;p}^p \right\}. 
\end{equs}
\end{lemma}

\begin{proof} On the one hand, we have that
\begin{equs}
 |\varphi_p(f_s) - \varphi_p(g_s)| & \leq 2^{p-1} \left(\|f_s\|_p^p + \|g_s\|_p^p\right),
\end{equs}
which implies that
\begin{equs}
 \left|\int_0^1 s^{\frac{\al p}{2}} \varphi_p(f_s) \frac{\dd s}{s} - \int_0^1 s^{\frac{\al p}{2}} \varphi_p(g_s) \frac{\dd s}{s}\right| & \leq 2^{p-1} \left(\|f\|_{-\al;p}^p + \|g\|_{-\al;p}^p\right). \label{eq:bnd_1}
\end{equs}
On the other hand, we notice that
\begin{equs}
 |\varphi_p(f_s) - \varphi_p(g_s)| &  \leq 2^{p-1} \int |f_s^p \mathbf{1}_{\{f_s\geq 0\}} - g_s^p \mathbf{1}_{\{g_s\geq 0\}}| \, \dd x
 \\
 & \quad + 2^{p-1} \int |f_s^p \mathbf{1}_{\{f_s< 0\}} - g_s^p \mathbf{1}_{\{g_s< 0\}}| \, \dd x.
\end{equs}
Then, using the fact that for $a,b\in \RR$,
\begin{equs}
 \left. 
 \begin{aligned} 
  & |a^p\mathbf{1}_{\{a\geq 0\}}-b^p\mathbf{1}_{\{b\geq 0\}}| \\
  & |a^p\mathbf{1}_{\{a< 0\}}-b^p\mathbf{1}_{\{b< 0\}}|
 \end{aligned}
 \right\} & \leq |a-b| \left(|a|^{p-1} + |a|^{p-2} |b| + \ldots + |a| |b|^{p-2} + |b|^{p-1}\right),  
\end{equs}
we easily see that
\begin{equs}
 \left|\int_0^1 s^{\frac{\al p}{2}} \varphi_p(f_s) \frac{\dd s}{s} - \int_0^1 s^{\frac{\al p}{2}} \varphi_p(g_s) \frac{\dd s}{s}\right|
 \lesssim 2^{p-1} \|f-g\|_{-\al;p} \left(\|f\|_{-\al;p}^{p-1} + \|g\|_{-\al;p}^{p-1}\right).
 \\
 \label{eq:bnd_2}
\end{equs}
Combining \eqref{eq:bnd_1} and \eqref{eq:bnd_2} we obtain the desired estimate. 
\end{proof}

\section{Order preservation for approximations}

\begin{proposition} \label{prop:order_pres} Let $f_1,f_2$ be smooth functions such that $f_1\preceq f_2$ and denote by $u_\eps(\cdot;f_i)$ the solution 
to 
\begin{equs}
 \begin{cases}
  & (\partial_t - \Delta) u_\eps = - \left(u^3_\eps - 3C u_\eps\right) + u_\eps+ \xi*\rho_\eps \\
  &  u_\eps|_{t=0} = f_i,
 \end{cases}
\end{equs}
for $i=1,2$, $C>0$ and $\rho_\eps$ a smooth mollifier for $\eps\in(0,1]$ . Then, we have that $u_\eps(t;f_1) \preceq u_\eps(t;f_2)$, for every $t\geq 0$.
\end{proposition}

\begin{proof} Let $v(t)= \ee^{-Kt} (u_\eps(t;f_2) - u_\eps(t;f_1))$ for some $K>0$. Then $v$ solves the following equation,
\begin{equs}
 \begin{cases}
  & (\partial_t - \Delta) v = -(K - 3C - 1) v - 3 \int_0^1 \big(\lambda u_\eps(\cdot;f_2) + (1-\lambda) u_\eps(\cdot;f_1) \big)^2 \, \dd \lambda \, v \\
  &  u_\eps|_{t=0} = f_2-f_1.
 \end{cases}
\end{equs}
Testing with $v_- := - (v\wedge 0)$ we obtain that
\begin{equs}
 & \frac{1}{2} \partial_t \|v_-\|_2^2 + \|\nabla v_-\|_2^2 
 \\
 & \quad = - (K - 3C - 1) \|v_-\|_2^2 
 - 3 \left \lng \int_0^1 \big(\lambda u_\eps(\cdot;f_2) + (1-\lambda) u_\eps(\cdot;f_1) \big)^2 \, \dd \lambda , v_-^2 \right\rng_{L^2}.
\end{equs}
Choosing $K\geq 3C + 1$ gives that
\begin{equs}
 \frac{1}{2} \partial_t \|v_-\|_2^2 \leq 0
\end{equs}
which in turn implies that 
\begin{equs}
 \|v_-(t)\|_2 \leq  \|v_-(0)\|_2,
\end{equs}
for every $t\geq 0$. Since $f_1 \preceq f_2$ we know that $v_-(0) = 0$, hence $\|v_-(t)\|_2 = 0$, for every $t\geq 0$, which completes
the proof. 
\end{proof}

\end{appendix}

\pdfbookmark{References}{references}
\addtocontents{toc}{\protect\contentsline{section}{References}{\thepage}{references.0}}

\bibliographystyle{plain}
\bibliography{synchr_bibliography}{}

\begin{flushleft}
\small \normalfont
\textsc{Benjamin Gess\\
Max--Planck--Institut f\"ur Mathematik in den Naturwissenschaften\\ 
04103 Leipzig, Germany\\
Faculty of Mathematics, University of Bielefeld\\
33615 Bielefeld, Germany}\\
\texttt{\textbf{benjamin.gess@mis.mpg.de}}
\end{flushleft}

\begin{flushleft}
\small \normalfont
\textsc{Pavlos Tsatsoulis\\
Max--Planck--Institut f\"ur Mathematik in den Naturwissenschaften\\
04103 Leipzig, Germany}\\
\texttt{\textbf{pavlos.tsatsoulis@mis.mpg.de}}
\end{flushleft}

\end{document}